\documentclass[11pt]{amsart}
\usepackage{amsmath, amssymb, amscd, amsthm, amsfonts, latexsym}
\usepackage{graphicx}
\usepackage{tikz-cd}
\usepackage{hyperref}
\usepackage{mathtools}
\usepackage{mathrsfs}
\usepackage{enumerate}
\usepackage{upgreek}
\usepackage{mathabx}
\usepackage{setspace}
\usepackage{color}
\usepackage{comment}
\usepackage{tikz}
\usepackage{tikz-cd}

\newtheorem{theorem}{Theorem}[section]

\newtheorem{definition}{Definition}
\newtheorem{remark}{Remark}
\newtheorem{corollary}{Corollary}[section]
\numberwithin{equation}{section}
 \def\@evenhead{\vbox{\hbox to \textwidth{\thepage\hfil\sl\leftmark\strut}\hrule}}
 \def\@oddhead{\vbox{\hbox to \textwidth{\rightmark\hfill\thepage\strut}\hrule}}

\def\Z{{\mathbb Z}}

\def\cd{\protect\operatorname{cd}}
\def\cat{\protect\operatorname{cat}}
\def\dcat{\protect\operatorname{dcat}}
\def\dTC{\protect\operatorname{dTC}}

\def\im{\protect\operatorname{im}}

\title{Small Value of cohomological dimension of group homomorphisms}
\author{Nursultan Kuanyshov}

\address{Nursultan Kuanyshov, Suleyman Demirel University, Kaskelen, Kazakhstan}
\email{nursultan.kuanyshov@sdu.edu.kz, kuanyshov.nursultan@gmail.com}



\begin{document}
\maketitle

\begin{abstract}
We give the characterisation of cohomological dimension of group homomorphisms. In particular, we show the analogue of the Stallings-Swan theorem in the case of group epimorphisms. 
\end{abstract}

\section{Introduction}

The concept of cohomological dimension $\mathrm{cd}(\Gamma)$ of a discrete group $\Gamma$ dates back to the mid-20th century and has played a central role in algebraic topology and geometric group theory. Foundational results in the subject are collected in the classical works of Bieri \cite{Bi} and Brown \cite{Br}.

The \emph{cohomological dimension of a group homomorphism} $\phi: \Gamma \to \Lambda$, denoted $\mathrm{cd}(\phi)$, was first proposed by Mark Grant in a 2013 MathOverflow discussion \cite{Gr}. While this notion naturally generalizes the classical cohomological dimension of groups, it remained largely unexplored until recent developments. Partial results appeared in \cite{Sc, DK}, but a systematic framework was provided only recently by Dranishnikov and De Saha \cite{DD}, who introduced homological and geometric tools suitable for studying $\cd(\phi)$.

In the group case, small values of cohomological dimension are strongly tied to structural properties. The Stallings-Swan theorem characterizes groups of cohomological dimension one as free, while the Eilenberg-Ganea theorem relates cohomological and geometric dimension, with the exceptional case $\cd(\Gamma) = 2$ forming the well-known Eilenberg-Ganea conjecture. However, both results fail in the homomorphism setting: joint work with Dranishnikov \cite{DK} constructed epimorphisms $\phi: \Gamma \to \Lambda$ with $\mathrm{cd}(\phi) = 1$ and $\mathrm{gd}(\phi) = 2$, and with $\mathrm{cd}(\phi) = 2$ and $\mathrm{gd}(\phi) = 3$. In \cite{DD}, it is shown that for any $n \geq 2$, there exist homomorphisms with $\mathrm{cd}(\phi) = n$ and $\mathrm{gd}(\phi) = n+1$, establishing that the divergence between cohomological and geometric dimensions persists at all levels by 1.

An important distinction between the classical and homomorphism settings is the behavior of groups with finite elements. For those groups, $\mathrm{cd}(\Gamma) = \infty$, necessitating the use of \emph{virtual cohomological dimension} to study them. In contrast, homomorphisms between finite groups can have \emph{finite} cohomological dimension. In this paper, we give explicit examples of such homomorphisms. As a result, there is no need to define the virtual cohomological dimension for group homomorphisms, marking a conceptual simplification relative to classical group cohomology.

Despite the failure of classical classification theorems, the study of small cohomological dimensions for homomorphisms still leads to rigidity and structural results. In this paper, we focus on \emph{group epimorphisms} $\phi: \Gamma \twoheadrightarrow \Lambda$ between groups, and we prove the following two theorems:

\medskip

\noindent\textbf{Theorem 1.3} \label{char. of group hom}
\emph{Let $\phi: \Gamma \twoheadrightarrow \Lambda$ be a group epimorphism. Then}
$$
\cd(\phi) = 0 \,\, \text{if and only if} \,\, \phi \,\, \text{is the trivial homomorphism.}
$$

\medskip

This result confirms that cohomological dimension zero imposes strong rigidity in the homomorphism setting, mirroring the classical case where $\mathrm{cd}(\Gamma) = 0$ if and only if $\Gamma$ is trivial.

Our second main result addresses the next level, $\mathrm{cd}(\phi) = 1$, and connects it with the \emph{Lusternik-Schnirelmann category of a homomorphism}, denoted $\mathrm{cat}(\phi)$, as introduced in \cite{DK}.

\medskip

\noindent\textbf{Theorem 2.}  
\emph{Let $\phi: \Gamma \twoheadrightarrow \Lambda$ be a group epimorphism. Then}
$$
\cat(\phi) = \cd(\phi) = 1 \,\, \text{if and only if} \,\,  \phi\,\, \text{factors through the free groups} 
$$
$$ \text{via epimorphisms} \, \,  \Gamma\to F_{n}\to \Lambda.
$$

\medskip

This theorem can be seen as a homomorphism-level analogue of the Stallings-Swan theorem. It shows that the condition $\cat(\phi)=\cd(\phi)=1$ captures a certain freeness - not necessarily of the group $\Gamma$ itself, but of the \emph{factorization} of the map $\phi$ through a free object in the category of groups. 

Our methods combine several tools: the \emph{Berstein-Schwarz class} is developed by Dranishnikov and Rudyak \cite{DR}; the \emph{projective resolution} constructed by Dranishnikov and De Saha \cite{DD}; and classical cohomological techniques from Brown's book \cite{Br}, particularly extension theory and dimension shifting.

In the paper, we use the notation $H^*(\Gamma, A)$ for the cohomology of a group $\Gamma$ with coefficient in $\Gamma$-module $A$. The cohomology groups of a space $X$ with the fundamental group $\Gamma$ we denote as $H^*(X;A)$. Thus, $H^*(\Gamma,A)=H^*(B\Gamma;A)$ where $B\Gamma=K(\Gamma,1)$, a classifying space (Eilenberg-MacLane space): a path-connected space such that $\pi_{1}(B\Gamma)=\Gamma$ and $\pi_{k}(B\Gamma)=0$ for all $k\neq 1.$

\section{Preliminaries}

We recall our tools and  classical numerical invariants such as Lusternik-Schnirelmann category, topological complexity, and sequential topological complexity in this section.

\subsection{Berstein-Schwarz cohomology class}

The Berstein-Schwarz class of a discrete group $\Gamma$ is the first obstruction $\beta_{\Gamma}$ to a lift of $B\Gamma=K(\Gamma,1)$ to the universal covering $E\Gamma$. 
Note that $\beta_{\Gamma}\in H^1(\Gamma,I(\Gamma))$ where $I(\Gamma)$ is the augmentation ideal of the group ring $\Z\Gamma$~\cite{Be},\cite{DR}.

\begin{theorem}[Universality~\cite{DR,Sch}]\label{universal}
For any cohomology class $\alpha\in H^k(\Gamma,L)$, there is a homomorphism of $\Gamma$-modules $I(\Gamma)^{\otimes k}\to L$ such that the induced homomorphism for cohomology takes $(\beta_{\Gamma})^k\in H^k(\Gamma,I(\Gamma)^{\otimes k})$ to $\alpha$,  where $I(\Gamma)^{\otimes k}=I(\Gamma)\otimes\dots\otimes I(\Gamma)$ and $(\beta_{\Gamma})^k=\beta_{\Gamma}\smile\dots\smile\beta_{\Gamma}$.
\end{theorem}

\subsection{Projective resolution of Dranishnikov and De Saha \cite{DD}}

Let $\Gamma$ be a discrete group. A projective resolution $P_{*}(\Gamma)$ of $\Z$ for the group $\Gamma$ is an exact sequence of projective $\Gamma-$ modules over $\Z$:
$$\cdots\to P_{k}(\Gamma)\to P_{k-1}(\Gamma)\to\cdots P_{2}(\Gamma)\to P_{1}\to \Z\Gamma\to \Z.$$

The cohomology of $\Gamma$ with the coefficients in a $\Gamma$-module M can be defined as homology of the cochain complexe $Hom_{\Gamma}(P_{*}(\Gamma),M)$.

We recall the {\em cohomological dimension} $\cd(\phi)$ of a group homomorphism $\phi:\Gamma\to\Lambda$ was introduced by Mark Grant~\cite{Gr} as the maximum of $k$ such that
there is a $\Lambda-$module $M$ with  the nonzero  induced  homomorphism $\phi^*:H^k(\Lambda,M)\to H^k(\Gamma,M)$.

The following Theorem of Dranishnikov and De Saha gives the characterization of group homomorphism in term of the projective resolutions. 

\begin{theorem}[\cite{DD}] Let $\phi:\Gamma\to\Lambda$ be a group homomorphism with $\cd(\phi)=k$ and let $\phi_{*}:(P_{*}(\Gamma),\partial_{*})\to (P_{*}(\Gamma),\partial'_{*})$ be the chain map between projective resolutions of $\Z$ for $\Gamma$ and $\Lambda$ induced by $\phi$. Then $\phi_{*}$ is chain homotopic to $\psi_{*}:P_{*}(\Gamma)\to P_{*}(\Lambda)$ with $\psi_{n}=0$ for $n>k.$
\end{theorem}

\subsection{LS-category and Sequential Topological Complexity}

Let $f:X\to Y$ be a map. Let $X^{r}$ and $Y^{r}$ be the Cartesian product of $r$ copies of $X$ and $Y$ respectively, i.e $X^{r}:=X\times\cdots\times X$ and $Y^{r}:=Y\times\cdots\times Y.$ Let us denote $f^{r}:=f\times\cdots\times f:X^{r}\to Y^{r}$ and elements of $X^{r}$ and $Y^{r}$ are vectors $\bar{x}=(x_{0},\cdots,x_{r-1})$ and $\bar{y}=(y_{0},\cdots,y_{r-1})$ respectively.

\begin{definition}
 The Lusternik-Schnirelmann category of the map f, $\cat(f)$, between X and Y topological spaces is the minimal number k such that X admits an open cover by k+1 open sets $U_{0},U_{1},..., U_{k}$ and restriction of f to each $U_{i}$ is null-homotopic. 
\end{definition}

\begin{definition} Let $f:X\to Y$ be the above map.
\begin{enumerate} 
\item A {\em sequential $f$-motion planner} on a subset $U\subset X^{r}$ is a map $f_{U}:U\to PY$ such that $f_{U}(\bar{x})(\frac{j}{r-1})=f_{U}(x_{0},x_{1},\cdots,x_{r-1})(\frac{j}{r-1})=f(x_{j})$ for all $j=0,\cdots,r-1.$

\item The {\em sequential topological complexity of map f}, denoted $TC_{r}(f),$ is the minimal number $k$ such that $X^{r}$ is covered by $k+1$ open sets $U_{0},\cdots,U_{k}$ on which there are sequential $f$-motion planners. If no such $k$ exists, we set $TC_{r}(f)=\infty.$
\end{enumerate}
\end{definition}
Note that if $r=2,$ we recover Scott's topological complexity for a map \cite{Sc}. Further, if map $f$ is identity on space $X,$ we get Rudyak's sequential topological complexity of space $X$ \cite{Ru}.

Due to the homotopy invariance of classifying spaces, there is a one-to-one correspondence between group homomorphism $\phi:\Gamma\to\Lambda$ and homotopy classes of maps $B\phi:B\Gamma\to B\Lambda$ that induce $\phi$ on the fundamental group. Hence, the following definition is well-defined and equivalent (see more details in \cite{Dr})

\begin{definition}
 Let $\phi:\Gamma\to\Lambda$ be a homomorphism.
 \begin{enumerate}
\item The LS-category of $\phi$, $\cat(\phi),$ is defined to be $\cat(B\phi).$
\item The geometric dimension $gd(\phi)$ of $\phi$ is defined as the minimal $k$ such that $B\phi$ can be deformed to the k-skeleton $B\Lambda^{(k)}$: 
$$gd(\phi)=\min\{n|B\phi \sim  f:B\Gamma\to B\Lambda^{(k)}\}$$
 \end{enumerate}
\end{definition}

\subsection{Ganea-Schwarz's approach to LS-category} Recall that an element of an iterated join $X_0*X_1*\cdots*X_k$ of topological spaces is a formal linear combination $t_0x_0+\cdots +t_k{x_k}$ of points $x_i\in X_i$ with $\sum t_i=1$, $t_i\ge 0$, in which all terms of the form $0x_i$ are dropped. Given fibrations $f_i:X_i\to Y$ for $i=0, ..., k$, the fiberwise join of spaces $X_0, ..., X_k$ is defined to be the space
\[
    X_0*_YX_1*_Y\cdots *_YX_k=\{\ t_0x_0+\cdots +t_k{x_k}\in X_0*\cdots *X_k\ |\ f_0(x_0)=\cdots =f_k(x_k)\ \}.
\]
The fiberwise join of fibrations $f_0, ..., f_n$ is the fibration 
\[
    f_0*_Y*\cdots *_Yf_k: X_0*_YX_1*_Y\cdots *_YX_k \longrightarrow Y
\]
defined by taking a point $t_0x_0+\cdots +t_k{x_k}$ to $f_i(x_i)$ for any $i$ such that $t_i \neq 0$. 

When $X_i=X$ and $f_i=f:X\to Y$ for all $i$  the fiberwise join of spaces is denoted by $*^{k+1}_YX$ and the fiberwise join of fibrations is denoted by $*_Y^{k+1}f$. 

For a path connected space X, we turn an inclusion of a point $i:x_{0}\to X$ into a fibration $p_{0}^{X}:G_{0}(X)\to X$, whose fiber is known to be the loop space $\Omega X$. The n-th Ganea space of X is
defined to be the space $G_{k}(X)=*_{X}^{k+1}G_{0}(X)$, while the k-th Ganea fibration $p_{k}^{X}:G_{k}(X)\to X$ is the fiberwise join $*_{X}^{k+1}p_{0}^{X}$. Then the fiber of $p_{0}^{X}$ is $*^{k+1}\Omega X$.

The following theorem give the Ganea-Shwarz characterization of LS-category of maps \cite{Sc, Sch, DK, CLOT}.

\begin{theorem}\label{GS}
If $f:X\to Y$ is a map between connected normal spaces, then $\cat(f)\leq k$ if and only if
there is a lift of f with respect to the fibration $p_{k}^{Y}:G_{k}(Y)\to Y$ admits a section.
\end{theorem}

\section{Proof of Main Results}

Given a homomorphism $\phi:\Gamma\to\Lambda$ by $\phi':\Gamma\to\im(\phi)$ we denote the restriction of $\phi$ from the codomain to its range. Joint work with Dranishnikov \cite{DK} established the following result, which tells it suffices to consider an epimorphisms instead of all homomorphisms. 

\begin{theorem}\cite{DK}\label{epimorphism}
    For any homomorphism $\phi:\Gamma\to\Lambda$, $$\cat(\phi)=\cat(\phi)$$ $$\cd(\phi)=\cd(\phi)$$
\end{theorem}

By Theorem \ref{epimorphism}, we assume that group homomorphism $\phi:\Gamma\to \Lambda$ is an epimorphism. We get the following characterization of group epimorphism.  

\begin{theorem}\label{char. of group hom}
Let $\phi:\Gamma\to \Lambda$ be an epimorphism of finitely generated groups $\Gamma, \Lambda$. Then $\cd(\phi)=0$ if and only if $\phi$ is the trivial homomorphism.    
\end{theorem}
\begin{proof}
Suppose $\phi$ is the trivial homomorphism. $\Lambda=0$ since $\phi$ is surjective. The trivial group has its classifying space homotopy equivalent to the point, so all reduced cohomology is zero. Hence, we obtain that $\phi^{*}:\tilde{H^{*}}(B\Lambda;M)\to \tilde{H^{*}}(B\Gamma;M)$ is the trivial homomorphism for all $\Lambda-$ modules M. By definition, this gives us that $\cd(\phi)=0.$

Note that $P_{*}(\Gamma)$, $\Z\Gamma,$ and $I(\Gamma)$ are $\Gamma-$ modules, but also $\Lambda-$ modules via $\phi.$

The Berstein-Schwarz class of $\Lambda$, $\beta_{\Lambda}\in H^{1}(B\Lambda;I(\Lambda))$,  is given by the cocycle $f_{\Lambda}:P_{1}(\Lambda)\to I(\Lambda),f_{\Lambda}(c')=\partial'_{1}(c')$ where $I(\Lambda)$ is the augmentation ideal of $\Z\Lambda$.

By contradiction, assume $\phi$ is not trivial epimorphism. Since $\phi_{*}$ is nonzero surjective, for each $c'\in P_{1}(\Lambda)$, there exists $c\in P_{1}(\Gamma)$ with $c'=\phi(c).$ Since the above diagram is commutative, we get the nonzero cocycle $f_{\Lambda}\phi_{1}\in Hom_{\Lambda}(P_{1}(\Gamma),I(\Lambda))$. Hence, this implies $\phi^{*}:H^{1}(B\Lambda; I(\Lambda))\to H^{1}(B\Gamma;I(\Lambda))$
is not trivial, i.e $\phi^{*}(\beta_{\Lambda})\neq 0.$ In other words, the cohomological class $\phi^{*}(\beta{\Lambda})$ is given by a nonzero cocycle $f_{\Lambda}\phi_{1}.$ This contradicts $\cd(\phi)=0.$ This implies $\Gamma$ is the trivial group, so $\phi$ is the trivial homomorphism.  
\end{proof}

\begin{corollary}\label{free group}
 Let $\phi:F_{n}\to F_{m}$ be an epimorphism between free groups with generators $n$ and $m$, respectively. Then $\cat(\phi)=\cd(\phi)=1$.  
\end{corollary}

\begin{proof}
 By Theorem \ref{char. of group hom}, it is clear that $\cd(\phi)\geq 1$ since $F_{m}$ is not trivial.
Since $\cd(\phi)\leq \cat(\phi)\leq \cat(F_{n})=1,$ we obtain that $\cat(\phi)=\cd(\phi)=1.$

\end{proof}

\begin{theorem}\label{EG}
 Let $\phi:\Gamma\to\Lambda$ be an epimorphism of the groups. Then $$\cat(\phi)=\cd(\phi)=1$$ if and only if $\phi$ factors through the free groups via epimorphisms $\Gamma\to F_{n}\to \Lambda.$   
\end{theorem}

\begin{proof}
By Theorem \ref{char. of group hom}, we always obtain $\cd(\phi)\geq 1$ since $\Lambda$ is not the trivial group. 

Now we suppose the epimorphism $\phi:\Gamma\to\Lambda$ factors through the free groups via epimorphisms $\Gamma\to F_{n}\to \Lambda.$ Since $\cd(\phi)\leq \cat(\phi)$ (see \cite{DK}), it suffices to show $\cat(\phi)\leq 1.$

This easily follows from the classical LS-category. Indeed, $$\cat(\phi)\leq \min\{\cat(h),\cat(g)\}\leq \cat(F_{n})=1,$$ where $h:B\Gamma\to BF_{n}$ and $g:BF_{n}\to B\Lambda$ are maps that induced epimorphisms $h_{*}:\Gamma\to F_{n}$ and $g_{*}:F_{n}\to \Lambda.$

Now conversely $\cat(\phi)=\cd(\phi)=1,$ then using the Ganea-Schwarz approach to LS-category of induced map $B\phi$, namely by Theorem \ref{GS} we obtain that the induced map $B\phi:B\Gamma\to B\Lambda$ can be lifted to the Ganea's space $G_{1}(B\Lambda).$ Note that $G_{1}(B\Lambda)$ is homotopy equivalent to 1-dimensional complex $\Sigma\Lambda$, the reduced suspension of $\Lambda$ [\cite{CLOT}, Example 1.61, page 27]. Then the epimorphism $\phi$ factors through a free group $F_{n}$ via epimorphisms $\Gamma\to F_{n}\to\Lambda.$    

\end{proof}

\begin{theorem}\label{finite}
For an  epimorphism $\phi:\mathbb Z_{p^4}\to\mathbb Z_{p^2}$,
$$
\cd(\phi)=2\ \ \ \text{and}\ \ \ \cat(\phi)=\infty.
$$
\end{theorem}
\begin{proof}
The proof of the inequality $\cd(\phi)\le 2$ is based on the following observation: {\em There is a degree 0 map $$f:L^3_{p^4}\to L^3_{p^2}$$ of lens spaces  that induces an epimorphism $f_*:\mathbb Z_{p^4}\to\mathbb Z_{p^2}$ of the fundamental groups.} 

Here is the construction of $f$. 
We consider two circles $S^1$ with $\mathbb Z_{p^2}$ and $\mathbb Z_{p^4}$ free actions.
Note that a $\mathbb Z_{p^4}$-equivariant map $\psi:S^1\to S^1$ between them has degree $p^{2}$. Then the join product $\psi\ast\psi:S^1\ast S^1\to S^1\ast S^1$ has degree $p^4$.
Hence it induces a map of the orbit spaces $\bar\psi:L^3_{p^4}\to L^3_{p^2}$ of degree $p^2$. Let $q_k:S^3\to L_{p^k}^3$ denote the projection onto the orbit space, were k=2,4. 

We define $f$ as the composition
$$
L^3_{p^4}=L^3_{p^4}\#S^3\to L^3_{p^4}\vee S^3\stackrel{1\vee r}\to L^3_{p^4}\vee S^3\stackrel{\bar\psi\vee q_2}\to L^3_{p^2}\vee L^3_{p^2}\stackrel{j}\to L^3_{p^2}
$$
where  $r:S^3\to S^3$ has degree -1 and $j$ identifies two copies of $L^3_{p^2}$.
Thus, $deg(f)=p^2-p^2=0$.

Claim 1. The induced homomorphism $f^*:H^3(L^3_{p^2};M)\to H^3(L^3_{p^4};M)$ is trivial for any $\mathbb Z_{p^2}$-module $M$.

Indeed, if $f^*(\alpha)\ne 0$ for $\alpha\in H^3(L_{p^2}^3;M)$ then by the Poincare duality with local coefficients $[L^3_{p^4}]\cap f^*(\alpha)\ne 0$. Since $f$ induces an epimorphism of the fundamental groups, the induced homomorphism for 0-homology $f_*:H_0(L^3_{p^4};M) \to H_0(L^3_{p^2};M)$ is an isomorphism. We obtain a contradiction 
$$0\ne f_*([L^3_{p^4}]\cap f^*(\alpha))=f_*([L^3_{p^4}])\cap\alpha= 0.$$

The map $f$ extends to a map $\phi:L^\infty_{p^4}\to L^\infty_{p^2}$.

Claim 2. The induced homomorphism $\phi^*:H^3(L^\infty_{p^2};M)\to H^3(L^\infty_{p^4};M)$ is trivial for any $\mathbb Z_{p^2}$-module $M$.

This follows from the fact that the inclusion homomorphism $H^k(X;M)\to H^k((X^{(k)};M)$ is injective for any CW-complex $X$ and a $\pi_1(X)$-module $M$.

The shift of dimension for group cohomology~\cite{Br} implies that $$\phi^*:H^k(L^\infty_{p^2};M)\to H^k(L^\infty_{p^4};M)$$ is trivial for all $k\ge 3$.
Thus, $\cd(\phi)\le 2$. Using integral cohomology one can check that $\cd(\phi)\ge 2$.

Next we show that $\cat(\phi)=\infty$. This follows from the fact that the reduced K-theory cup-length of $\phi$ is unbounded. By the Atiyah's computations~\cite{AS}, $$K(B\mathbb Z_{p^k})=\mathbb Z[[\eta_k]]/(\eta_k^{p^k}-1)$$ where $\eta_k$ is the pull-back of the canonical line bundle $\eta$ over $\mathbb CP^\infty$ under the inclusion
$\mathbb Z_{p^k}\to S^1$ and $\mathbb Z[[x]]$ is the ring of formal series. The reduced K-theory of $B\mathbb Z_{p^k}$ is the subring generated by $\eta_k-1$.
The induced homomorphism $$\phi^*:K(B\mathbb Z_{p^2})\to K(B\mathbb Z_{p^4})$$  takes $\eta_2$ to $\eta_4^{p^{2}}$. Then it takes $(\eta_2-1)^m$ to 
$(\eta_4^{p^{2}}-1)^m$. To see that $(\eta_4^{p^{2}}-1)^m\ne 0$ in $\mathbb Z[[\eta_4]]/(\eta_4^{p^4}-1)$ for all $m$ it suffices to show that it is not zero in
$\mathbb Z_p[[\eta_4]]/(\eta_4^{p^4}-1)$ for $m=p^{2}k+1$. Note that $$(\eta_4^{p^{2}}-1)^{p^{2}k+1}=\eta_4^{p^{2}}-1\ne 0$$ in $\mathbb Z_p[[\eta_4]]/(\eta_4^{p^4}-1)$.
This implies $\cat(\phi)\ge m$ for any $m$ (see Proposition 5.1 in~\cite{Dr}).
\end{proof}

\begin{remark}
    In the proof of Theorem \ref{finite}, we can see that the proof works if we change the epimorphism $\phi:\Z_{p^{2}}\to\Z_{p}$ between different groups $\Z_{p^{2}},\Z_{p}$. Thus, we get many examples of such epimorphisms with $\cd(\phi)=2$ and $\cat(\phi)=\infty$.   
\end{remark}

\section{Application of Main Results}

Due to the homotopy invariance of classifying spaces, there is a one-to-one correspondence between group homomorphism $\phi:\Gamma\to\Lambda$ and homotopy classes of maps $B\phi:B\Gamma\to B\Lambda$ that induce $\phi$ on the fundamental group. Hence, the following definition is well-defined. By Theorem \ref{epimorphism}, we assume $\phi$ is an epimorphisms in this section. 

\begin{definition}
 Let $\phi:\Gamma\to\Lambda$ be a homomorphism.
 \begin{enumerate}
\item The LS-category of $\phi$, $\cat(\phi),$ is defined to be $\cat(B\phi).$
\item The sequential topological complexity of $\phi$, $TC_{r}(\phi),$ is defined to be $TC_{r}(B\phi).$
 \end{enumerate}
\end{definition}

By Theorem \ref{finite}, we can see that the gap between the LS-category and the cohomological dimension of group homomorphisms can be arbitrarily large. This theorem answers the question of \cite{Ku3} partially in the case of finite groups. This is because the previous known gaps were 1 (see \cite{DK, DD, Gr}). 

Furthermore, by Theorem \ref{finite} we have the group epimorphism that does not satisfy $TC_{r}(\phi)\leq (r+1)\cd(\phi)$ since $\infty=\cat(\phi)\leq TC_{r}(\phi)$ and $(r+1)\cd(\phi)=2(r+1)$. Thus, there are no analogue results of Rudyak \cite{Ru} result, $$r\cd(\Gamma)\leq TC_{r}(\Gamma)\leq (r+1)\cd(\Gamma),$$ in the context of group homomorphisms.

Furthermore, Dranishnikov and Jauhari \cite{DJ} introduced distributional topological complexity and LS-category for topological spaces and maps, denoted as $\dcat$ for LS-category and $\dTC$ for topological complexity. They proved that distributional invariants are lower bound for classical ones, namely LS-category and topological complexity. 

Since cohomological dimensions are also lower bound for LS-category, the natural question is if cohomological dimension and distributional LS-category are the same or not. In here, we gave the answer saying these two numerical invariants are not comparable. 

For example, take p=2, we get $\cd(\phi)=2$ since the proof of Theorem \ref{finite} works for this epimorphism $\phi:\Z_{p^{2}}\to\Z_{p}$ as well while $\dcat(\phi)=1$ since $\dcat(\Z_{2})=1$ (see more details \cite{DJ}). Other case, take covering map $p:S^{n}\to RP^{n}$ from n-dimensional sphere to n-dimensional projective space. Since the map $p$ is not null-homotopic, we get $\dcat(p)=1$ while $\cd(p_{*})=0$ since $S^{n}$ is the simply-connected for $n\geq 2.$

An analogue of the Eilenberg-Ganea result, $\cat(\Gamma)=\cd(\Gamma)$, in the setting of group homomorphisms fails in general, but also holds for many cases of finitely generated, torsion-free groups such as free, Abelian, nilpotent, virtually nilpotent, and almost nilpotent groups (see \cite{Sc, Ku1, Ku2, Ku3}). Extending these lists for classical groups is still a difficult task since each above results rely on special property of each group.   

By Theorem \ref{char. of group hom} and Theorem \ref{EG}, we can extend lists of groups where the analogue of the Eilenberg-Ganea result holds. 

\begin{theorem}\label{one relator}
   Given an epimorphism $\phi:\Gamma\to\Lambda$ between finitely generated, one relator groups $\Gamma,\Lambda$. Then $\cat(\phi)=\cd(\phi)$  
\end{theorem}

\begin{proof}

It is well known that cohomological dimension of one relator groups are at most 2 (see \cite{Br}). Indeed, it is two because its classifying space is 2-dimensional complex. By Theorem \ref{char. of group hom}, $\cd(\phi)\geq 1$ since $\phi$ is nonzero epimorphisms. Thus, there are two cases:

Case 1: If given the epimorphism factor through free groups, then by Theorem \ref{EG} we get $\cat(\phi)=\cd(\phi)=1$, which confirms the analogue of the Eilenberg-Ganea result.

Case 2: Otherwise, $\cd(\phi)\geq 2$ since $\phi$ does not factor through via free groups. Since $\cd(\Gamma)=2$, we get that $\cat(\phi)\leq 2$. Therefore, combining two inequalities we get $\cat(\phi)=\cd(\phi)=2$, which again confirms the anologue result of Eilenberg-Ganea.

\end{proof}

\begin{corollary}
   Given an epimorphism $\phi:\Gamma\to\Lambda$ of the fundamental groups of orientable closed surfaces M, N, i.e $\Gamma=\pi_{1}(M), \Lambda=\pi_{1}(N)$ Then $\cat(\phi)=\cd(\phi)$   
\end{corollary}

\begin{proof}
    It follows from Theorem \ref{one relator} since the fundamental groups of orientable closed surfaces are one relator groups. 
\end{proof}

\begin{corollary}
   Given an epimorphism $\phi:\Gamma\to\Lambda$ of Baumslag-Solitar groups.
 Then $\cat(\phi)=\cd(\phi)$   
\end{corollary}

\begin{proof}
    It follows from Theorem \ref{one relator} since the Baumslag-Solitar groups are one relator groups.
\end{proof}

\def\bibname{\vspace*{-30mm}{

\end{document}